\documentclass[letterpaper, 10 pt, conference]{ieeeconf}

\IEEEoverridecommandlockouts
\usepackage{cite}

\usepackage{amsthm,amsmath,amssymb,amsfonts}
\usepackage{bm}
\usepackage{algorithmic}
\usepackage[ruled, noline, linesnumbered]{algorithm2e}
\usepackage{graphicx}
\usepackage{xcolor}

\usepackage{array}
\usepackage{multirow}

\makeatletter
\let\MYcaption\@makecaption
\makeatother

\usepackage[font=footnotesize]{subcaption}

\makeatletter
\let\@makecaption\MYcaption
\makeatother

\usepackage{textcomp}
\def\BibTeX{{\rm B\kern-.05em{\sc i\kern-.025em b}\kern-.08em
T\kern-.1667em\lower.7ex\hbox{E}\kern-.125emX}}

\usepackage{lipsum}


\theoremstyle{plain}

\newtheorem{lemma}{Lemma}
\newtheorem{corollary}{Corollary}

\theoremstyle{definition}

\def\({\left(}
\def\){\right)}
\def\[{\left[}
\def\]{\right]}

\def\abf{{\bf a}}
\def\bbf{{\bf b}}
\def\dbf{{\bf d}}
\def\Fbf{{\bf F}}  
\def\Gbf{{\bf G}}  
\def\Kbf{{\bf K}}  
\def\Mbf{{\bf M}}  
\def\Nbf{{\bf N}}  
\def\Pbf{{\bf P}}  
\def\Qbf{{\bf Q}}  
\def\Rbf{{\bf R}}  
\def\Sbf{{\bf S}}  
\def\Tbf{{\bf T}}  
\def\Ubf{{\bf U}}  \def\ubf{{\bf u}}
\def\Vbf{{\bf V}}  
\def\Wbf{{\bf W}}  \def\wbf{{\bf w}}
\def\xbf{{\bf x}}
\def\Ybf{{\bf Y}}  \def\ybf{{\bf y}}
\def\Zbf{{\bf Z}}  \def\zbf{{\bf z}}
\def\Deltabf{\bm{\Delta}}  \def\deltabf{\bm{\delta}}
\def\etabf{{\bm{\eta}}}
\def\Phibf{\bm{\Phi}}
\def\Psibf{\bm{\Psi}}
\def\Ccal{\mathcal{C}}  
\def\Dcal{\mathcal{D}}  
\def\Hcal{\mathcal{H}}  
\def\Kcal{\mathcal{K}}  
\def\Rcal{\mathcal{R}}  

\newif\ifshowWriterComment
\newcommand\writercomment[3]{\expandafter}

\newcommand{\hatSFpair}{\{\hat{\Phibf}_\xbf, \hat{\Phibf}_\ubf\}}
\newcommand{\hatOFqple}{\{\hat{\Phibf}_{\xbf\xbf}, \hat{\Phibf}_{\ubf\xbf}, \hat{\Phibf}_{\xbf\ybf}, \hat{\Phibf}_{\ubf\ybf}\}}

\newcommand{\norm}[1]{\left\lVert #1 \right\rVert}

\def\mat#1{\begin{bmatrix}#1\end{bmatrix}}
\def\cor#1{Corollary~\ref{cor:#1}}
\def\fig#1{Fig.~\ref{fig:#1}}
\def\subfig#1#2{Fig.~\ref{fig:#1}(\subref{subfig:#1-#2})}
\def\lem#1{Lemma~\ref{lem:#1}}
\def\sec#1{Section~\ref{sec:#1}}
\def\eqn#1{\eqref{eqn:#1}}

\def\st{{\rm s.t.}}
\def\OptConsSep{&&\quad}

\newcommand\OptCons[3]{
&\ #1
\ifx\\#2\\ \else \OptConsSep #2 \fi%
\ifx\\#3\\ \nonumber \else \label{eqn:#3} \fi%
}

\newcommand{\OptMinN}[2]{
\begin{alignat*}{2}
\min\ &\ #1 \\
\st\ #2
\end{alignat*}
}

\def\xx{\xbf\xbf}  \def\xu{\xbf\ubf}  \def\xy{\xbf\ybf}
\def\ux{\ubf\xbf}  \def\uu{\ubf\ubf}  \def\uy{\ubf\ybf}
\def\yx{\ybf\xbf}  \def\yu{\ybf\ubf}  \def\yy{\ybf\ybf}

\def\Rtru{\Rbf(\Deltabf)}
\def\Stru{\Sbf(\Deltabf)}
\def\Rnm{\hat{\Rbf}}
\def\Snm{\hat{\Sbf}}

\title{\LARGE \bf A General Approach to Robust Controller Analysis and Synthesis}

\author{Shih-Hao Tseng
\thanks{Shih-Hao Tseng is with the Division of Engineering and Applied Science, California Institute of Technology, Pasadena, CA 91125, USA.  Email: {\tt\small shtseng@caltech.edu}}
}

\showWriterCommenttrue

\begin{document}

\maketitle
\thispagestyle{empty}
\pagestyle{empty}

\bstctlcite{IEEE_BSTcontrol}

\begin{abstract}
Robust controller synthesis attracts reviving research interest, driven by the rise of learning-based systems where uncertainty and perturbation are ubiquitous. Facing an uncertain situation, a robustly stabilizing controller should maintain stability while operating under a perturbed system deviating from its nominal specification. There have been numerous results for robust controller synthesis in multiple forms and with various goals, including $\mu$-synthesis, robust primal-dual Youla, robust input-output, and robust system level parameterizations. However, their connections with one another are not clear, and we lack a general approach to robust controller analysis and synthesis.

To serve this purpose, we derive robust stability conditions for general systems and formulate the general robust controller synthesis problem. The conditions hinge on the realization-stability lemma, a recent analysis tool that unifies existing controller synthesis methods.
Not only can the conditions infer a wide range of existing robust results, but they also lead to easier derivations of new ones. Together, we demonstrate the effectiveness of the conditions and provide a unified approach to robust controller analysis and synthesis.

\end{abstract}

\section{Introduction}\label{sec:introduction}

Robust controller synthesis is an old topic with renewed interest, especially after the recent flourishing of learning-based controllers \cite{tsiamis2020sample, dean2020robust}.
In essence, it studies how the synthesized controller can still stabilize the plant in the presence of system \emph{perturbation}, also referred to as \emph{uncertainty}. For model-free controller synthesis tasks like learning-based data-driven control, such uncertainty is fundamental and inevitable as the plant is uncertain and subject to exploration. As a result, a robust controller is crucial to ensure stability while the system is learning.

As shown in the well-celebrated paper \cite{doyle1978guaranteed}, robustness cannot be inherited for free from simply optimizing some objective. Instead, one would need to explicitly enforce robustness constraints. As to what robustness constraints should be enforced, is itself an active area of research with fruitful results in diverse forms \cite{doyle1982analysis, doyle1985structured, zheng2020sample, matni2017Scalable,anderson2019system,boczar2018finite}.
Although robust results on controller synthesis vary in their forms, they usually focus on mitigating one of the following two concerns -- plant uncertainty or controller perturbation.

Due to the estimation precision, dimension limit, or pliant nature, the plant model could differ from its true dynamics. To deal with this uncertainty, one category of robust results aims to synthesize a controller that can stabilize a set of plants, such as $\mu$-synthesis \cite{doyle1982analysis, doyle1985structured, zhou1998essentials}, robust primal-dual Youla parameterization \cite{niemann2002reliable}, and robust input-output parameterization (IOP) \cite{zheng2020sample}.
On the other hand, even with an exact plant model, realization of the synthesized controller may still deviate from its desired form because of resolution restriction. As such, another category of robust results ensures a perturbed controller realization can still stabilize the plant, e.g., robust system level synthesis (SLS) \cite{matni2017Scalable,anderson2019system,boczar2018finite}.

These results are derived via various analysis tools targeting distinct settings. It is not straightforward to see how they relate with one another and how one method may be applicable for a different setting. As a consequence, most robust results are taught, learned, and applied in a case-by-case manner. Moreover, in addition to the two concerns above, one can easily imagine some compound scenarios where both the plant and controller are subject to perturbation. We would then wonder how to deal with diverse perturbation scenarios systematically. In particular, we are interested in a unified approach to robust controller analysis and synthesis.

\subsection{Contributions and Organization}

This paper provides such a unified approach through the robust stability conditions for general systems. The approach is built upon the \emph{realization} abstraction, which describes a (closed-loop) system by the linear maps among its internal signals.\footnote{We adopt the terminology in \cite{tseng2020deployment,tsengsubsynthesis} that distinguishes ``realizations'' from ``implementations,'' where the former refers to the block diagrams (mathematical expressions) and the latter is reserved for the physical architecture consisting of computation, memory, and communication units.}
As both plant and controller are included in a realization, the uncertainty to any of them is deemed a perturbation of the realization, thereby unifying diverse robustness concerns into one coherent form.
The realization-centric perspective also allows us to apply the \emph{realization-stability lemma} \cite{tseng2021realization} to derive the robust stability condition for general perturbations. We also specialize our result for additively perturbed realizations. The robust stability condition leads to the formulation of the general robust controller synthesis problem, and we demonstrate how to derive existing results in the literature using the condition. In addition, we show how new robust results for output-feedback SLS and IOP can be obtained easily from the condition.

The paper is organized as follows. We first briefly review the realization-stability lemma and its accompanying transformation technique in~\sec{R-S_lemma}, which allows us to derive robust stability conditions and formulate the general robust controller synthesis problem in~\sec{robust_synthesis}.
Leveraging the conditions derived in~\sec{robust_synthesis}, we show in \sec{existing_results} how to unify existing results, including $\mu$-synthesis \cite{doyle1982analysis, doyle1985structured, zhou1998essentials}, robust primal-dual Youla parameterization \cite{niemann2002reliable}, robust IOP \cite{zheng2020sample}, and robust SLS \cite{matni2017Scalable,anderson2019system,boczar2018finite}. Further, we demonstrate in~\sec{robust_of_SLS} how the results in~\sec{robust_synthesis} allow us to present new robust results for output-feedback SLS and IOP. Although our approach is promising in unifying many existing results, there are still some robust results in the literature beyond our scope. Thus, we discuss how this work relates with other robust results in~\sec{discussion}. Finally, we conclude the paper in~\sec{conclusion}.

\subsection{Notation}

\def\Rp{\Rcal_{p}}
\def\RHinf{\Rcal\Hcal_{\infty}}

Let $\Rp$
and $\RHinf$ denote the set of proper
and stable proper transfer matrices, respectively, all defined according to the underlying setting, continuous or discrete.
Lower- and upper-case letters (such as $x$ and $A$) denote vectors and matrices respectively, while bold lower- and upper-case characters and symbols (such as $\ubf$ and $\Rbf$) are reserved for signals and transfer matrices. We denote by $I$ and $O$ the identity and all-zero matrices (with dimensions defined according to the context).

\section{Realization-Stability Lemma}\label{sec:R-S_lemma}
We briefly summarize the realization-stability lemma \cite{tseng2021realization} below, which suggests that the product of the realization and internal stability matrices of a system is identity. To begin with, we consider a closed-loop linear system with internal state $\etabf$ and external disturbance $\dbf$. The internal state $\etabf$ includes \emph{all} signals in the system. As such, instability of the system can be inferred from the unboundedness of $\etabf$.

The realization and internal stability matrices characterize the linear system through closed-loop and open-loop relationships between $\etabf$ and $\dbf$.
The \emph{realization matrix} $\Rbf$ captures the linear relationships among the internal signals. With additive external disturbance, as shown in \subfig{perturbed-realization}{real-system}, $\Rbf$ is defined as
\begin{align}
\etabf = \Rbf \etabf + \dbf.
\label{eqn:R}
\end{align}
On the other hand, treating $\dbf$ as the input and $\etabf$ the output, the \emph{internal stability matrix} (or \emph{stability matrix} for short) $\Sbf$ describes the system from an open-loop perspective by
\begin{align}
\etabf = \Sbf \dbf.
\label{eqn:S}
\end{align}

\begin{figure}
\centering
\subcaptionbox{True realization $\Rtru$.\label{subfig:perturbed-realization-real-system}}{\includegraphics{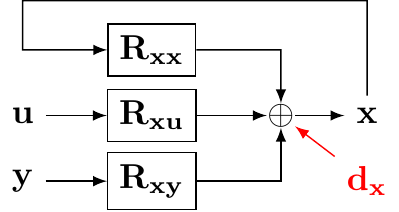}}\hfill
\subcaptionbox{Realization perturbed by $\Deltabf$.\label{subfig:perturbed-realization-general-perturbation}}{\includegraphics{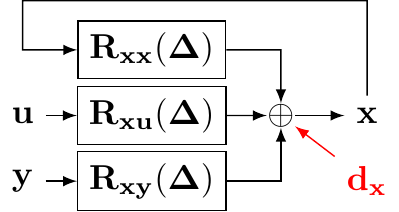}}\\
\vspace*{0.5\baselineskip}
\subcaptionbox{Additive Perturbation.\label{subfig:perturbed-realization-additive-perturbation}}{\includegraphics{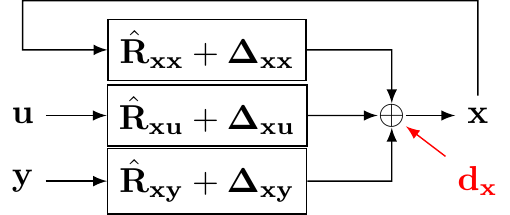}}
\caption{The realization matrix $\Rbf$ describes each signal as a linear combination of the signals in the closed-loop system and the external disturbance $\dbf$.
We denote by $\Rtru$ the realization matrix under perturbation $\Deltabf$ and by $\Rnm$ the nominal realization matrix without perturbation. In the following figures of realizations, we omit drawing the additive disturbance $\dbf$ for simplicity.}
\label{fig:perturbed-realization}
\end{figure}

$\Rbf$ and $\Sbf$ are related through the realization-stability lemma below derived in \cite{tseng2021realization}.
\begin{lemma}[Realization-Stability]\label{lem:R-S}
Let $\Rbf$ be the realization matrix and $\Sbf$ be the internal stability matrix, we have
\begin{align*}
(I - \Rbf) \Sbf = \Sbf(I - \Rbf) = I.
\end{align*}
\end{lemma}
We remark that \lem{R-S} does not guarantee the existence of $\Rbf$ and $\Sbf$. Instead, \lem{R-S} asserts the relationships between them when they both exist.

A useful technique introduced along with the realization-stability lemma in \cite{tseng2021realization} is the \emph{transformation} of the external disturbance. A transformation $\Tbf$ expresses the external disturbance $\dbf$ over a different basis $\wbf$:
\begin{align*}
\dbf = \Tbf \wbf.
\end{align*}
An invertible transformation not only changes the basis but also derives the equivalent system with realization $\Rbf_{eq}$ and stability $\Sbf_{eq}$ by
\begin{align*}
\Rbf_{eq} = I - \Tbf^{-1}(I-\Rbf),\quad\quad
\Sbf_{eq} = \Sbf\Tbf.
\end{align*}
\lem{R-S} implies that if two systems have the same realization matrix, they have the same stability matrix. Consequently, if we can transform a system to a certain target realization, the transformed stability matrix is the target stability matrix. This technique then enables simpler proofs of equivalence among parameterizations in \cite{tseng2021realization}.

\section{Robust Stability and Controller Synthesis}\label{sec:robust_synthesis}
In this section, we derive the condition for robust stability analysis using \lem{R-S}. The condition then allows us to formulate the general robust controller synthesis problem.

\subsection{Robust Stability and Additive Perturbations}
Consider a system perturbed according to some uncertain parameter $\Deltabf \in \Dcal$ as in \subfig{perturbed-realization}{general-perturbation}, where $\Dcal$ is the uncertainty set. Denote by $\Rtru$ its realization matrix and by $\Stru$ the corresponding stability matrix.
By \lem{R-S}, the perturbed realization and stability matrices satisfy
\begin{align*}
(I - \Rtru) \Stru = \Stru (I - \Rtru) = I.
\end{align*}
Also, the perturbed system is \emph{robustly stable} if and only if the open-loop system from the external disturbance $\dbf$ to the internal state $\etabf$ is stable under all uncertain parameter $\Deltabf \in \Dcal$. In other words, we require the stability matrix $\Stru$ to obey
\begin{align}
\Stru \in \RHinf, \quad \forall \Deltabf \in \Dcal.
\label{eqn:robust-S-general}
\end{align}

Suppose the system is subject to additive perturbation\footnote{Some papers refer the additive perturbation here as ``multiplicative fault'' \cite{niemann2002reliable} since it appears in the equations as a multiplier of a signal.} and its realization matrix $\Rtru$ can be expressed as
\begin{align}
\Rtru = \Rnm + \Deltabf
\label{eqn:perturbed-R}
\end{align}
where $\Rnm$ is the \emph{nominal realization}, as shown in \subfig{perturbed-realization}{additive-perturbation}. Define the \emph{nominal stability} $\Snm$ as the stability matrix accompanying the nominal realization $\Rnm$, we can express $\Stru$ in terms of $\Snm$ and the perturbation $\Deltabf$ as follows.

\begin{lemma}[Stability under Additive Perturbation]\label{lem:perturbed-S}
Let $\Rnm$ be the nominal realization matrix and $\Snm$ be the nominal internal stability matrix. Suppose the system realization $\Rtru$ is subject to additive perturbation \eqn{perturbed-R},
the corresponding stability $\Stru$ is given by
\begin{align*}
\Stru = \Snm (I - \Deltabf \Snm)^{-1} = (I - \Snm\Deltabf)^{-1} \Snm.
\end{align*}
\end{lemma}

\begin{proof}
\lem{R-S} implies $(I - \Rnm) \Snm = I$.
Therefore,
\begin{align*}
(I - \Rtru) \Snm = (I - \Rnm - \Deltabf) \Snm = I - \Deltabf \Snm.
\end{align*}
As a result, \lem{R-S} suggests
\begin{gather*}
(I - \Rtru) \Snm (I - \Deltabf \Snm)^{-1} = I \nonumber \\
\Rightarrow\quad
\Stru = \Snm (I - \Deltabf \Snm)^{-1}.
\end{gather*}
Similarly, we can also derive $\Stru = (I - \Snm \Deltabf)^{-1} \Snm$ from $\Snm(I-\Rnm) = I$, and the lemma follows.
\end{proof}

We can interpret the resulting stability matrix $\Sbf(\Deltabf)$ in \lem{perturbed-S} as the nominal stability $\Snm$ with a feedback path $\Deltabf$ as in \fig{perturbed-stability}. From this perspective, an additive perturbation to the realization results in a feedback path in the stability.

\begin{figure}
\centering
\includegraphics{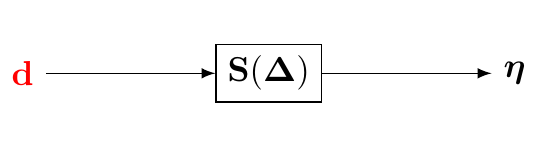}
\includegraphics{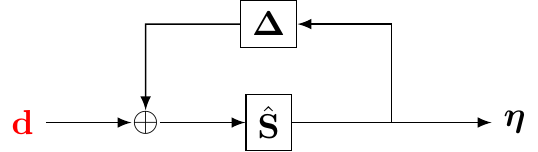}
\caption{The stability matrix $\Sbf$ maps external disturbance $\dbf$ to internal state $\etabf$. We denote by $\Stru$ the stability matrix under perturbation $\Deltabf$ and by $\Snm$ the nominal stability matrix without perturbation. When the perturbation is additive to the realization, it results in a feedback path for the stability.}
\label{fig:perturbed-stability}
\end{figure}

According to \lem{perturbed-S} and condition \eqn{robust-S-general}, to ensure robust stability of the system, we need
\begin{align}
\Snm (I - \Deltabf \Snm)^{-1} \in \RHinf, \quad \forall \Deltabf \in \Dcal.
\label{eqn:robust-S}
\end{align}

\subsection{General Formulation for Robust Controller Synthesis}
We can now generalize the general controller synthesis problem in \cite{tseng2021realization} to its robust version using condition \eqn{robust-S-general}:
\OptMinN{
g(\Rtru,\Stru,\Dcal)
}{
&\ (I-\Rtru)\Stru = \Stru&&(I-\Rtru) = I \\
\OptCons{}{\forall \Deltabf \in \Dcal}{}\\
\OptCons{\Rtru_{\abf\bbf} \in \Rp}{\forall \Deltabf \in \Dcal, \abf \neq \bbf}{}\\
\OptCons{\Stru \in \RHinf}{\forall \Deltabf \in \Dcal}{}\\
\OptCons{(\Rtru,\Stru) \in \Ccal}{\forall \Deltabf \in \Dcal}{}
}
where $\Ccal$ represents some additional constraints on the realization and stability. In particular, for a system subject to additive perturbation, the problem can be reformulated as
\OptMinN{
g(\Rnm,\Snm,\Dcal)
}{
\OptCons{(I-\Rnm)\Snm = \Snm(I-\Rnm) = I}{}{}\\
\OptCons{\Rtru_{\abf\bbf} = (\Rnm + \Deltabf)_{\abf\bbf} \in \Rp}{\forall \Deltabf \in \Dcal, \abf \neq \bbf}{}\\
\OptCons{\Stru = \Snm (I - \Deltabf \Snm)^{-1} \in \RHinf}{\forall \Deltabf \in \Dcal}{}\\
\OptCons{(\Rtru,\Stru) \in \Ccal}{\forall \Deltabf \in \Dcal}{}
}
This formulation is general as it can describe not only the robust controller synthesis problem for a given uncertainty set $\Dcal$ but also the stability margin problem like $\mu$-synthesis \cite{doyle1982analysis, doyle1985structured, zhou1998essentials}, where $\Dcal$ is itself a variable to ``maximize''.

Despite its generality, solving this general formulation is challenging in general, and the major obstacle is to ensure the perturbed stability for all $\Deltabf \in \Dcal$. Except for some computationally tractable cases as those listed in \sec{existing_results}, enforcing the robust constraints involves dealing with semi-infinite programming when the $\Dcal$ has infinite cardinality. For those cases, one may instead enforce chance constraints and adopt sampling-based techniques as in \cite{calafiore2005uncertain, erdougan2006ambiguous,tseng2016random}.

\section{Corollaries: Existing Robust Results}\label{sec:existing_results}
We unite the existing robust results under \lem{R-S} and \lem{perturbed-S}, including $\mu$-synthesis \cite{doyle1982analysis, doyle1985structured, zhou1998essentials}, primal-dual Youla parameterization \cite{niemann2002reliable}, input-output parameterization \cite{zheng2020sample}, and system level synthesis \cite{matni2017Scalable,anderson2019system,boczar2018finite}.

\begin{figure}
\centering
\includegraphics{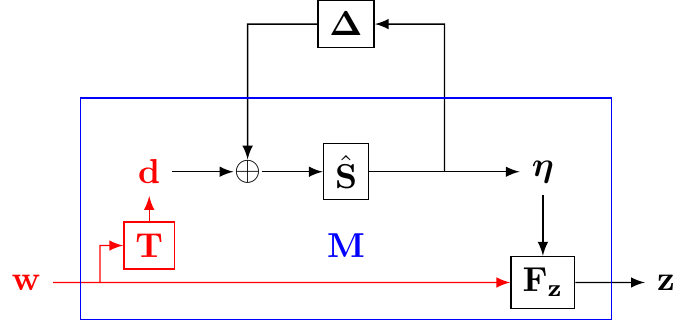}
\caption{For additive perturbation $\Deltabf$, the matrix $\Mbf$ in the $\mu$-synthesis is essentially the nominal stability matrix $\Snm$ wrapped by transformation $\Tbf$ and the output transfer matrix $\Fbf_{\zbf}$.}
\label{fig:mu-synthesis}
\end{figure}

\subsection{$\mu$-Synthesis}

$\mu$-synthesis studies the quantity ${\rm det}(I-\Mbf\Deltabf)$ \cite{doyle1982analysis, doyle1985structured, zhou1998essentials}, where the matrix $\Mbf \in \RHinf$ maps some external disturbance $\wbf$ to some output $\zbf$. The $\mu$ function, or the structured singular value, is the inverse size of the ``smallest'' structured perturbation that destabilizes $(I-\Mbf\Deltabf)^{-1}$.
Further, the original papers claim that it is always possible to choose $\Mbf$ so that $\Deltabf$ is block-diagonal. As to why ${\rm det}(I-\Mbf\Deltabf)$ is critical and why $\Deltabf$ could be formed block-diagonal, \lem{perturbed-S} provides intuitive answers.

Suppose the realization is subject to additive perturbation $\Deltabf$. We first show that $\Mbf$ is a wrapped nominal stability matrix $\Snm$.
Since $\etabf$ contains all the internal signals, the output $\zbf$ can be expressed as a linear map $\Fbf_\zbf$ of $\etabf$ and the external input $\wbf$. Meanwhile, $\wbf$ can be mapped to the external disturbance $\dbf$ through a transformation $\Tbf$. Therefore, the matrix $\Mbf$ can be expressed as in \fig{mu-synthesis}. Here the $\Stru$ is expanded by \lem{perturbed-S} as in \fig{perturbed-stability}.

Since $\Deltabf$ is an additive perturbation on the realization, we can create dummy signals to the realization so that each signal has at most one perturbed input/output. Equivalently, $\Deltabf$ is structured such that there is at most one perturbation block at each row/column. Hence, $\Deltabf$ can be permuted to take a block-diagonal form.

On the other hand, when $\wbf = \dbf$ and $\zbf = \etabf$, we have $\Mbf = \Snm \in \RHinf$. By \lem{perturbed-S}, the perturbed stability is
\begin{align*}
\Stru = \Snm(I-\Snm\Deltabf)^{-1} = \Mbf(I-\Mbf\Deltabf)^{-1},
\end{align*}
which is not stable if $I-\Mbf\Deltabf$ is not invertible, or ${\rm det}(I-\Mbf\Deltabf) = 0$.

\begin{figure}
\centering
\parbox{0.425\linewidth}{\includegraphics[scale=1]{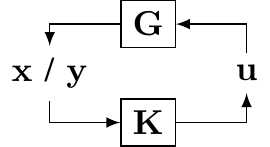}}
\parbox{0.425\linewidth}{\begin{gather*}
(I - \Rbf)\etabf = \\
\mat{I & -\Gbf\\
-\Kbf & I}
\mat{\xbf \text{ / } \ybf \\ \ubf}
\end{gather*}}
\caption{The realization with plant $\Gbf$ and controller $\Kbf$. The internal signals include state $\xbf$ (or measurement $\ybf$) and control $\ubf$.}
\label{fig:realization-G-K}
\end{figure}

\subsection{Robust Primal-Dual Youla Parameterization}
Youla parameterization is based on the doubly coprime factorization. Suppose
\begin{align}
\mat{
\Mbf_l & -\Nbf_l\\
-\Vbf_l & \Ubf_l
}\mat{
\Ubf_r & \Nbf_r\\
\Vbf_r & \Mbf_r
} = I
\label{eqn:coprime-factorization}
\end{align}
where both matrices are in $\RHinf$, $\Mbf_l$ and $\Mbf_r$ are both invertible in $\RHinf$.
\cite[Chapter 3, Theorem 4.2]{tay1998high} generalizes the primal Youla parameterization $\Qbf$ (see \cite[Theorem 5.6]{zhou1998essentials}) to also parameterize the stabilizable plants $\Gbf$ using dual Youla parameter $\Pbf$.\footnote{To avoid the confusion with the stability matrix $\Sbf$, we express the dual Youla parameter by $\Pbf$ instead.}
Here we present a modern rewrite of \cite[Chapter 3, Theorem 4.2]{tay1998high}:

\begin{corollary}\label{cor:primal-dual-Youla}
For the realization in \fig{realization-G-K}, let $\Gbf$ and $\Kbf$ be parameterized by
\begin{align*}
\Gbf =&\ (\Nbf_r - \Ubf_r \Pbf)(\Mbf_r-\Vbf_r \Pbf)^{-1}\\
\Kbf =&\ (\Vbf_r - \Mbf_r \Qbf)(\Ubf_r -\Nbf_r\Qbf)^{-1}
\end{align*}
where $\Pbf, \Qbf \in \RHinf$ are the dual and primal Youla parameters. The system is internally stable if and only if the following $(\Pbf,\Qbf)$ realization is internally stable:
\begin{center}
\includegraphics[scale=1]{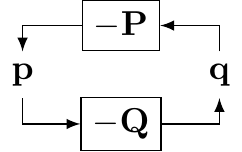}
\end{center}
\end{corollary}

\begin{proof}
Similar to the approach in \cite{tseng2021realization}, we transform the coprime factorization \eqn{coprime-factorization} by
\begin{align*}
\Tbf^{-1} =&
\mat{
(\Mbf_l-\Pbf\Vbf_l)^{-1} & O\\
O & (\Ubf_l -\Qbf\Nbf_l)^{-1} \\
}
\mat{
I & \Pbf\\
\Qbf & I
}\\
\Tbf =&
\mat{
I & \Pbf\\
\Qbf & I
}^{-1}
\mat{
\Mbf_l-\Pbf\Vbf_l & O\\
O & \Ubf_l -\Qbf\Nbf_l \\
}
\end{align*}
where
\begin{align*}
\mat{
I & \Pbf\\
\Qbf & I
}^{-1}
= \mat{
I & -\Pbf\\
-\Qbf & I
}
\mat{
(I - \Pbf\Qbf)^{-1} & O\\
O & (I - \Qbf\Pbf)^{-1}
}.
\end{align*}
The transformation leads to
\begin{align*}
& \Tbf^{-1}
\mat{
\Mbf_l & -\Nbf_l\\
-\Vbf_l & \Ubf_l
}
\mat{
\Ubf_r & \Nbf_r\\
\Vbf_r & \Mbf_r
}\Tbf\\
=& \mat{I & -\Gbf\\
-\Kbf & I}
\mat{
\Ubf_r & \Nbf_r\\
\Vbf_r & \Mbf_r
}\Tbf = I.
\end{align*}
By \lem{R-S}, the stability matrix of the system is
\begin{align*}
\mat{
\Ubf_r & \Nbf_r\\
\Vbf_r & \Mbf_r
}\Tbf
\end{align*}
which should be in $\RHinf$ to ensure internal stability. By assumption, the coprime factorization matrix, $\Pbf$, and $\Qbf$ are all in $\RHinf$. We then only need to ensure $\Tbf \in \RHinf$, which requires
\begin{align*}
\mat{
I & \Pbf\\
\Qbf & I
}^{-1} \in \RHinf.
\end{align*}
Treating the matrix inverse as a stability matrix, its corresponding realization is exactly the $(\Pbf,\Qbf)$ realization. Consequently, the whole system is internally stable if and only if the $(\Pbf,\Qbf)$ realization is internally stable, which concludes the proof.
\end{proof}

Leveraging the dual Youla parameterization, \cite{niemann2002reliable} proposes to embed the perturbation $\Deltabf$ on the plant $\Gbf$ into $\Pbf$, as such
\begin{align}
\Gbf(\Deltabf) =&\ (\Nbf_r - \Ubf_r \Pbf(\Deltabf))(\Mbf_r-\Vbf_r \Pbf(\Deltabf))^{-1}.
\label{eqn:Youla-plant}
\end{align}
\cite{niemann2002reliable} then suggests following result.

\begin{corollary}
Consider a perturbed Youla parameterization as in \cor{primal-dual-Youla} with the plant \eqn{Youla-plant}. The system is internally stable if
\begin{align}
(I - \Qbf\Pbf(\Deltabf))^{-1} \in \RHinf.
\end{align}
\end{corollary}

\begin{proof}
As shown in the proof of \cor{primal-dual-Youla}, the system is internally stable if and only if
\begin{align*}
&\mat{
I & \Pbf(\Deltabf)\\
\Qbf & I
}^{-1}\\
=& \mat{
I & -\Pbf(\Deltabf)\\
-\Qbf & I
}
\mat{
(I - \Pbf(\Deltabf)\Qbf)^{-1} & O\\
O & (I - \Qbf\Pbf(\Deltabf))^{-1}
}
\end{align*}
is in $\RHinf$. Since $\Pbf(\Deltabf), \Qbf \in \RHinf$ and
\begin{align*}
(I - \Pbf(\Deltabf)\Qbf)^{-1} = I + \Pbf (I - \Qbf\Pbf(\Deltabf))^{-1}\Qbf,
\end{align*}
we know $(I - \Qbf\Pbf(\Deltabf))^{-1} \in \RHinf$ implies $(I - \Pbf(\Deltabf)\Qbf)^{-1} \in \RHinf$, which concludes the proof.
\end{proof}

\subsection{Robust Input-Output Parameterization (IOP)}
\def\DeltaG{\Deltabf_{\Gbf}}

For a system with realization as in \fig{realization-G-K}, the internally stabilizing controllers can be parameterized by the input-output parameterization (IOP) as follows \cite{furieri2019input}:
\begin{align}
\mat{I & -\Gbf}
\mat{\Ybf & \Wbf\\ \Ubf & \Zbf}
=&
\mat{I & O},\nonumber \\
\mat{\Ybf & \Wbf\\ \Ubf & \Zbf}
\mat{-\Gbf\\ I}
=&
\mat{O\\I},\nonumber \\
\Ybf, \Ubf, \Wbf, \Zbf \in&\ \RHinf
\label{eqn:IOP-constraints}
\end{align}
\cite[Theorem C.2]{zheng2020sample} presents the following robust result of IOP.

\begin{corollary}\label{cor:robust-IOP}
For the realization in \fig{realization-G-K}, let $\Kcal_\epsilon$ be the set of robustly stabilizing controllers defined by
\begin{align*}
\Kcal_\epsilon = \left\lbrace
\Kbf : \Kbf \text{ internally stabilizes } \Gbf(\DeltaG), \forall \DeltaG \in \Dcal_\epsilon
\right\rbrace
\end{align*}
where $\Gbf(\Deltabf_{\Gbf}) = \Gbf + \DeltaG$ and $\Dcal_\epsilon = \left\lbrace \DeltaG : \norm{\DeltaG}_{\infty} < \epsilon \right\rbrace$. Then $\Kcal_\epsilon$ is parameterized by $\{\hat{\Ybf}, \hat{\Wbf}, \hat{\Ubf}, \hat{\Zbf} \}$ that satisfies \eqn{IOP-constraints} and
\begin{align*}
\norm{\hat{\Ubf}}_{\infty} \leq \epsilon^{-1},
\end{align*}
and the controller is given by $\hat{\Kbf} = \hat{\Ubf} \hat{\Ybf}^{-1}$
\end{corollary}

\begin{proof}
Let
\begin{align*}
\Snm = \mat{
\hat{\Ybf} & \hat{\Wbf}\\
\hat{\Ubf} & \hat{\Zbf}
}, \quad\quad
\Deltabf = \mat{
O & \DeltaG \\
O & O
}.
\end{align*}
\lem{perturbed-S} implies that the system is internally stable if and only if
\begin{align*}
\Stru = \Snm(I - \Deltabf \Snm)^{-1} \in \RHinf
\end{align*}
for all $\Deltabf \in \Dcal_\epsilon$. Since $O \in \Dcal_\epsilon$, we need $\Snm \in \RHinf$ (which leads to the first three conditions, see \cite{tseng2021realization}) and
\begin{align*}
&(I - \Deltabf \Snm)^{-1}
=
\mat{
I - \DeltaG \hat{\Ubf} & - \DeltaG \hat{\Zbf}\\
O & I
}^{-1}\\
=&
\mat{
(I - \DeltaG \hat{\Ubf})^{-1} & (I - \DeltaG \hat{\Ubf})^{-1} \DeltaG \hat{\Zbf} \\
O & I
} \in \RHinf.
\end{align*}
Since $\DeltaG$ has a bounded $H_{\infty}$ norm, we know $\DeltaG \in \RHinf$. By the small gain theorem, $(I - \DeltaG \hat{\Ubf})^{-1} \in \RHinf$ for all $\DeltaG \in \Dcal_\epsilon$ if and only if $\norm{\hat{\Ubf}}_{\infty} \leq \epsilon^{-1}$, which concludes the proof.
\end{proof}

\subsection{Robust System Level Synthesis (SLS)}

\renewcommand{\paragraph}[1]{\vspace*{0.5\baselineskip}\noindent\textbf{#1}}
\def\DeltaSLS{\Deltabf^{\rm SLS}}

System level synthesis (SLS) synthesizes internally stabilizing controllers using system level parameterization (SLP). Recent studies investigate robust controller synthesis problem with perturbed SLP under both state-feedback and output-feedback settings. We show below how to derive those results using \lem{perturbed-S}.

\begin{figure}
\centering
\parbox{0.425\linewidth}{\includegraphics[scale=1]{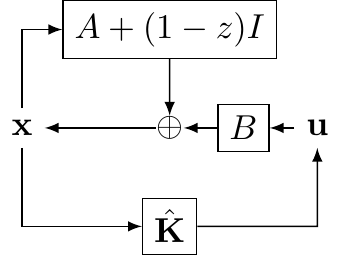}}
\parbox{0.425\linewidth}{\begin{gather*}
(I - \Rbf)\etabf = \\
\mat{zI-A & -B\\
-\hat{\Kbf} & I}
\mat{\xbf \\ \ubf}
\end{gather*}}
\caption{The realization of a state-feedback system with nominal controller $\hat{\Kbf}$. The internal signals are state $\xbf$ and control $\ubf$.}
\label{fig:realization-state-feedback}
\end{figure}

\begin{figure}
\centering
\parbox{0.475\linewidth}{\includegraphics[scale=1]{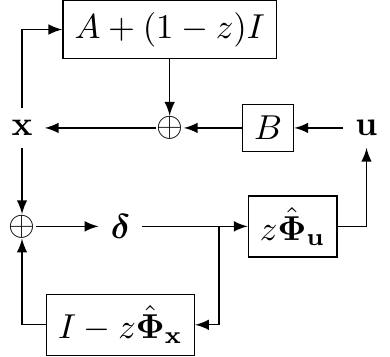}}
\parbox{0.5\linewidth}{\begin{gather*}
(I - \Rbf)\etabf = \\
\mat{zI-A & -B & O\\
O & I & -z\hat{\Phibf}_\ubf\\
-I & O & z\hat{\Phibf}_\xbf}
\mat{\xbf \\ \ubf \\ \deltabf}
\end{gather*}}
\caption{The robust state-feedback SLS realization using the SLP $\hatSFpair$. The internal signals are state $\xbf$, control $\ubf$, and estimated disturbance $\delta$.}
\label{fig:realization-sf-SLS}
\end{figure}

\paragraph{State-Feedback:} The following robust result from \cite[Theorem 2]{matni2017Scalable} and \cite[Theorem 4.3]{anderson2019system} examines the perturbed state-feedback SLP.

\begin{corollary}\label{cor:robust-sf-SLS}
For the system with realization as in \fig{realization-state-feedback}, let $(\hat{\Phibf}_\xbf, \hat{\Phibf}_\ubf, \DeltaSLS)$ be a solution to
\begin{align*}
\mat{zI - A & -B}
\mat{\hat{\Phibf}_\xbf\\ \hat{\Phibf}_\ubf} = I + \DeltaSLS,\quad  \hat{\Phibf}_\xbf, \hat{\Phibf}_\ubf \in z^{-1}\RHinf.
\end{align*}
Then, the controller realization
\begin{align*}
\hat{\deltabf}_\xbf = \xbf - \hat{\xbf}, \quad
\ubf = z \hat{\Phibf}_\xbf\hat{\deltabf}_\xbf, \quad
\hat{\xbf} = (z \hat{\Phibf}_\xbf - I)\hat{\deltabf}_\xbf
\end{align*}
internally stabilizes the system if and only if $(I + \DeltaSLS)^{-1}$ is stable. Furthermore, the actual system responses achieved are given by
\begin{align*}
\mat{\xbf \\ \ubf} =
\mat{\hat{\Phibf}_\xbf\\ \hat{\Phibf}_\ubf} (I + \DeltaSLS)^{-1} \dbf_\xbf.
\end{align*}
\end{corollary}

\begin{proof}
By definition, $(\hat{\Phibf}_\xbf, \hat{\Phibf}_\ubf, \DeltaSLS)$ is also a solution to
\begin{align*}
\mat{zI-A & -B\\
-\hat{\Kbf} & I}
\mat{\hat{\Phibf}_\xbf & \Sbf_{\xu} \\
\hat{\Phibf}_\ubf & \Sbf_{\uu}
} =
\mat{I + \DeltaSLS & O \\
O & I
}.
\end{align*}
Therefore,
by \lem{R-S}, the stability matrix for the realization in \fig{realization-state-feedback} is
\begin{align*}
&\mat{\hat{\Phibf}_\xbf & \Sbf_{\xu} \\
\hat{\Phibf}_\ubf & \Sbf_{\uu}
}
\mat{I + \DeltaSLS & O \\
O & I
}^{-1}\\
=& \mat{\hat{\Phibf}_\xbf(I + \DeltaSLS)^{-1} & \Sbf_{\xu} \\
\hat{\Phibf}_\ubf(I + \DeltaSLS)^{-1} & \Sbf_{\uu}
},
\end{align*}
which derives the system response.
As shown in \cite{tseng2021realization}, the nominal controller $\hat{\Kbf}$ is internally stabilizing if and only if
\begin{align*}
\hat{\Phibf}_\xbf(I + \DeltaSLS)^{-1}, \hat{\Phibf}_\ubf(I + \DeltaSLS)^{-1} \in z^{-1}\RHinf.
\end{align*}
In other words, since $\hat{\Phibf}_\xbf, \hat{\Phibf}_\ubf \in z^{-1}\RHinf$, $\Kbf$ is internally stabilizing if and only if $(I + \DeltaSLS)^{-1} \in \RHinf$. The internal stability of the proposed controller can then be showed by transforming the realization in \fig{realization-state-feedback} to the one in \fig{realization-sf-SLS} as done in \cite{tseng2021realization}.
\end{proof}

\begin{figure}
\centering
\parbox{0.45\linewidth}{\includegraphics[scale=1]{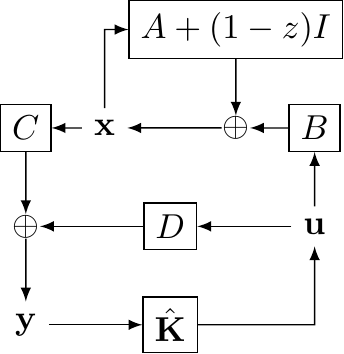}}
\parbox{0.45\linewidth}{\begin{gather*}
(I - \Rbf)\etabf = \\
\mat{zI-A & -B & O\\
O & I & -\hat{\Kbf} \\
-C & -D & I}
\mat{\xbf \\ \ubf \\ \ybf}
\end{gather*}}
\caption{The realization of an output-feedback system with nominal controller $\hat{\Kbf}$. The internal state $\etabf$ consists of state $\xbf$, control $\ubf$, and measurement $\ybf$ signals.}
\label{fig:realization-output-feedback}
\end{figure}

\paragraph{Output-Feedback:} For output-feedback systems as in \fig{realization-output-feedback}, the internally stabilizing controllers can be parameterized by the output-feedback SLP as follows \cite{wang2019system,anderson2019system}.
\begin{subequations}\label{eqn:of-SLS-constraints}
\begin{align}
\mat{
zI - A & -B
}
\mat{
\Phibf_{\xx} & \Phibf_{\xy}\\
\Phibf_{\ux} & \Phibf_{\uy}
}
=&
\mat{I & O},
\label{eqn:of-SLS-constraints-1}
\\
\mat{
\Phibf_{\xx} & \Phibf_{\xy}\\
\Phibf_{\ux} & \Phibf_{\uy}
}
\mat{
zI - A \\ -C
}
=&
\mat{I \\ O}, \\
\Phibf_{\xx}, \Phibf_{\xy}, \Phibf_{\ux} \in z^{-1}\RHinf,&\ \quad \Phibf_{\uy} \in \RHinf,
\label{eqn:of-SLS-constraints-3}
\end{align}
\end{subequations}

The two corollaries below are from \cite[Lemma 4.3, Corollary 4.4]{boczar2018finite}, which concern the controller resulting from perturbed \eqn{of-SLS-constraints}.

\begin{corollary}\label{cor:robust-of-SLS-condition}
For the system $\Gbf$ with realization as in \fig{realization-output-feedback} with $D=O$, let $\hatOFqple$ satisfy \eqn{of-SLS-constraints-1}, \eqn{of-SLS-constraints-3}, and
\begin{align}
\mat{
\hat{\Phibf}_{\xx} & \hat{\Phibf}_{\xy}\\
\hat{\Phibf}_{\ux} & \hat{\Phibf}_{\uy}
}
\mat{
zI - A \\ -C
}
=&
\mat{I + \DeltaSLS_1 \\ \DeltaSLS_2}
\label{eqn:robust-of-SLS-constraint}
\end{align}
and let the system response be given by
\begin{align*}
&\mat{
\Phibf_{\xx}(\DeltaSLS) & \Phibf_{\xy}(\DeltaSLS)\\
\Phibf_{\ux}(\DeltaSLS) & \Phibf_{\uy}(\DeltaSLS)
}\\
=&
\mat{
(I + \DeltaSLS_1)^{-1} & O\\
-\DeltaSLS_2 (I + \DeltaSLS_1)^{-1} & I
}
\mat{
\hat{\Phibf}_{\xx} & \hat{\Phibf}_{\xy} \\
\hat{\Phibf}_{\ux} & \hat{\Phibf}_{\uy}
}
\end{align*}
where by assumption $(I + \DeltaSLS_1)^{-1}$ exists and is in $\RHinf$. Then $\{\Phibf_{\xx}(\DeltaSLS)$, $\Phibf_{\ux}(\DeltaSLS)$, $\Phibf_{\xy}(\DeltaSLS)$, $\Phibf_{\uy}(\DeltaSLS)\}$ satisfies \eqn{of-SLS-constraints} for $\Gbf$. Furthermore, suppose $C$ is subject to an additive disturbance $\Deltabf_C$, i.e.,
\begin{align*}
C(\Deltabf_C) = C + \Deltabf_C,
\end{align*}
and
\begin{align*}
\DeltaSLS_1 = -\hat{\Phibf}_{\xy}\Deltabf_C,\quad\quad
\DeltaSLS_2 = -\hat{\Phibf}_{\xy}\Deltabf_C.
\end{align*}
We then have $\hatOFqple$ also satisfies \eqn{of-SLS-constraints} for the perturbed system.
\end{corollary}
\begin{corollary}\label{cor:robust-of-SLS-controller-condition}
Suppose $\hatOFqple$ satisfies the conditions in \cor{robust-of-SLS-condition}. Then, the controller
\begin{align*}
\hat{\Kbf} = \hat{\Phibf}_{\uy} - \hat{\Phibf}_{\ux} \hat{\Phibf}_{\xx}^{-1} \hat{\Phibf}_{\xy}
\end{align*}
stabilizes the system and achieves the closed-loop system response if and only if $(I + \DeltaSLS_1)^{-1} \in \RHinf$.
\end{corollary}

We now prove these two corollaries below by investigating $\Stru$.

\begin{proof}
By definition, $\hatOFqple$ is also a solution to
\begin{align*}
&\mat{
\hat{\Phibf}_{\xx} & \Sbf_{\xu} & \hat{\Phibf}_{\xy}\\
\hat{\Phibf}_{\ux} & \Sbf_{\uu} & \hat{\Phibf}_{\uy}\\
\Sbf_{\yx} & \Sbf_{\yu} & \Sbf_{\yy}
}
\mat{zI-A & -B & O\\
O & I & -\hat{\Kbf} \\
-C & -D & I}\\
=&
\mat{
I + \DeltaSLS_1 & O & O \\
\DeltaSLS_2 & I & O\\
O & O & I
}
\end{align*}
where $D = O$. Therefore, using the same transformation technique in \cite{tseng2021realization}, we can derive the nominal controller $\hat{\Kbf}$ in~\cor{robust-of-SLS-controller-condition} from
\begin{align*}
\mat{
\hat{\Phibf}_{\xx} & \hat{\Phibf}_{\xy}\\
\hat{\Phibf}_{\ux} & \hat{\Phibf}_{\uy}
}
\mat{zI-A & -B \\
-C & \hat{\Kbf}^{-1}}
=
\mat{
I + \DeltaSLS_1 & O \\
\DeltaSLS_2 & I
}.
\end{align*}
Also, by \lem{R-S}, the stability matrix for the realization in \fig{realization-output-feedback} is
\begin{align*}
&\mat{
\Phibf_{\xx}(\DeltaSLS) & \Sbf_{\xu}(\DeltaSLS) & \Phibf_{\xy}(\DeltaSLS)\\
\Phibf_{\ux}(\DeltaSLS) & \Sbf_{\uu}(\DeltaSLS) & \Phibf_{\uy}(\DeltaSLS)\\
\Sbf_{\yx}(\DeltaSLS) & \Sbf_{\yu}(\DeltaSLS) & \Sbf_{\yy}(\DeltaSLS)
}\\
=&
\mat{
I + \DeltaSLS_1 & O & O \\
\DeltaSLS_2 & I & O\\
O & O & I
}^{-1}
\mat{
\hat{\Phibf}_{\xx} & \Sbf_{\xu} & \hat{\Phibf}_{\xy}\\
\hat{\Phibf}_{\ux} & \Sbf_{\uu} & \hat{\Phibf}_{\uy}\\
\Sbf_{\yx} & \Sbf_{\yu} & \Sbf_{\yy}
}\\
=&
\mat{
(I + \DeltaSLS_1)^{-1} & O & O \\
-\DeltaSLS_2(I + \DeltaSLS_1)^{-1} & I & O\\
O & O & I
}
\mat{
\hat{\Phibf}_{\xx} & \Sbf_{\xu} & \hat{\Phibf}_{\xy}\\
\hat{\Phibf}_{\ux} & \Sbf_{\uu} & \hat{\Phibf}_{\uy}\\
\Sbf_{\yx} & \Sbf_{\yu} & \Sbf_{\yy}
},
\end{align*}
which derives the system response. Since
\begin{align*}
&\mat{
zI - A & -B
}
\mat{
\hat{\Phibf}_{\xx} & \hat{\Phibf}_{\xy}\\
\hat{\Phibf}_{\ux} & \hat{\Phibf}_{\uy}
}
\mat{
zI - A \\ -C
}\\
=&
\mat{
zI - A & -B
}
\mat{I + \DeltaSLS_1 \\ \DeltaSLS_2}\\
=&
\mat{I & O}
\mat{zI - A \\ -C},
\end{align*}
we have $(zI - A)\DeltaSLS_1 = B \DeltaSLS_2$. Therefore,
\begin{align*}
& \mat{
zI - A & -B
}
\mat{
\Phibf_{\xx}(\DeltaSLS) & \Phibf_{\xy}(\DeltaSLS)\\
\Phibf_{\ux}(\DeltaSLS) & \Phibf_{\uy}(\DeltaSLS)
}\\
=&
\mat{
zI - A & -B
}
\mat{
I + \DeltaSLS_1 & O \\
\DeltaSLS_2 & I
}^{-1}
\mat{
\hat{\Phibf}_{\xx} & \hat{\Phibf}_{\xy}\\
\hat{\Phibf}_{\ux} & \hat{\Phibf}_{\uy}
}\\
=&
\mat{
zI - A & -B
}
\mat{
\hat{\Phibf}_{\xx} & \hat{\Phibf}_{\xy}\\
\hat{\Phibf}_{\ux} & \hat{\Phibf}_{\uy}
}
\end{align*}

Also, \eqn{robust-of-SLS-constraint} and $z\hat{\Phibf}_{\ux}, \hat{\Phibf}_{\uy} \in \RHinf$ imply
\begin{align*}
\hat{\Phibf}_{\ux}(zI-A)-\hat{\Phibf}_{\uy}C = \DeltaSLS_2,
\end{align*}
and hence $\DeltaSLS_2 \in \RHinf$. Along with $(I + \DeltaSLS_1)^{-1} \in \RHinf$, we know that $\{\Phibf_{\xx}(\DeltaSLS)$, $\Phibf_{\ux}(\DeltaSLS)$, $\Phibf_{\xy}(\DeltaSLS)$, $\Phibf_{\uy}(\DeltaSLS)\}$ satisfies \eqn{of-SLS-constraints} through the same proof in \cite{tseng2021realization}.

When $C$ is perturbed, we have $\Rtru = \Rnm + \Deltabf$ where
\begin{align*}
\Deltabf = \mat{
O & O & O\\
O & O & O\\
\Deltabf_C & O & O\\
}
\end{align*}
and the result follows from \lem{perturbed-S}.

Since we have shown $\Stru \in \RHinf$ if and only if $(I + \DeltaSLS_1)^{-1}$ above, the nominal controller $\hat{\Kbf}$ stabilizes the system under the same condition, and \cor{robust-of-SLS-controller-condition} follows.
\end{proof}

\section{Robust Output-Feedback SLS and IOP}\label{sec:robust_of_SLS}
\lem{R-S} and \lem{perturbed-S} allow us to easily derive robust results by extending the nominal results. As examples, we derive the following robust results for output-feedback SLS that generalize \cor{robust-of-SLS-condition} and a condition for robust IOP.

\begin{corollary}\label{cor:robust-of-SLS-general}
For output-feedback systems as in \fig{realization-output-feedback} with arbitrary $D$, let $\hatOFqple$ satisfy \eqn{of-SLS-constraints}. Consider an additive perturbation $\Deltabf$ on the plant such that
\begin{align*}
A(\Deltabf) = A + \Deltabf_A,\quad&\quad B(\Deltabf) = B + \Deltabf_B,\\
C(\Deltabf) = C + \Deltabf_C,\quad&\quad D(\Deltabf) = D + \Deltabf_D.
\end{align*}
where $\Deltabf_A, \Deltabf_B, \Deltabf_C, \Deltabf_D \in \RHinf$. Then the nominal controller
\begin{align*}
\hat{\Kbf} = \hat{\Kbf}_0 \left( I + D \hat{\Kbf}_0\right)^{-1},
\end{align*}
where $\hat{\Kbf}_0 = \hat{\Phibf}_{\uy} - \hat{\Phibf}_{\ux} \hat{\Phibf}_{\xx}^{-1} \hat{\Phibf}_{\xy}$,
internally stabilizes the perturbed system if and only if $\Psibf \in \RHinf$ where
\begin{align*}
\Psibf = \left( I - \mat{
\Deltabf_A & \Deltabf_B\\
\Deltabf_C & \Deltabf_D
}
\mat{
\hat{\Phibf}_{\xx} & \hat{\Phibf}_{\xy}\\
\hat{\Phibf}_{\ux} & \hat{\Phibf}_{\uy}
} \right)^{-1}.
\end{align*}
\end{corollary}

\begin{proof}
Since $\hatOFqple$ satisfies \eqn{of-SLS-constraints}, we have the nominal stability $\Snm \in \RHinf$ and the nominal controller $\hat{\Kbf}$ is given by \cite[Corollary 5]{tseng2021realization}.

Given the realization in \fig{realization-output-feedback}, the perturbation $\Deltabf$ can be expressed as
\begin{align*}
\Deltabf = \mat{
\Deltabf_A & \Deltabf_B & O\\
O & O & O\\
\Deltabf_C & \Deltabf_D & O
}.
\end{align*}
Since $\Snm \in \RHinf$, \lem{perturbed-S} requires $(I-\Deltabf\Snm)^{-1} \in \RHinf$, or the inverse of
\begin{align}
I - \mat{
\Deltabf_A & \Deltabf_B & O\\
O & O & O\\
\Deltabf_C & \Deltabf_D & O
}
\mat{
\hat{\Phibf}_{\xx} & \Snm_{\xu} & \hat{\Phibf}_{\xy}\\
\hat{\Phibf}_{\ux} & \Snm_{\uu} & \hat{\Phibf}_{\uy}\\
\Snm_{\yx} & \Snm_{\yu} & \Snm_{\yy}
}
\label{eqn:robust-of-SLS-condition}
\end{align}
should be in $\RHinf$. Computing the inverse of \eqn{robust-of-SLS-condition} is equivalent to computing
\begin{align*}
\mat{
\Psibf & -\Psibf
\mat{
\Deltabf_A & \Deltabf_B\\
\Deltabf_C & \Deltabf_D
}
\mat{\Snm_{\xu}\\ \Snm_{\uu}}\\
O & I
}.
\end{align*}
We know $\Deltabf_A, \Deltabf_B, \Deltabf_C, \Deltabf_D, \Snm_{\xu}, \Snm_{\uu} \in \RHinf$.
Therefore, $\hat{\Kbf}$ still internally stabilizes the perturbed plant if and only if $\Psibf \in \RHinf$, which concludes the proof.
\end{proof}

We can extend \cor{robust-of-SLS-general} to provide an SLS version of \cor{robust-IOP}.

\begin{corollary}
For output-feedback systems $\Gbf(\Deltabf)$ as in \fig{realization-output-feedback} perturbed by structured $\Deltabf$ as in \cor{robust-of-SLS-general}, let $\Kcal_\epsilon$ be the set of robustly stabilizing controllers defined by
\begin{align*}
\Kcal_\epsilon = \{ \Kbf : \Kbf \text{ internally stabilizes } \Gbf(\Deltabf), \forall \Deltabf \in \Dcal_\epsilon \}
\end{align*}
where
\begin{align*}
\Dcal_\epsilon = \{
\Deltabf \text{ structured as in \cor{robust-of-SLS-general}} :
\norm{\Deltabf}_{\infty} < \epsilon
\}.
\end{align*}
Then $\Kcal_\epsilon$ is parameterized by $\hatOFqple$ that satisfies \eqn{of-SLS-constraints} and
\begin{align*}
\norm{
\mat{
\hat{\Phibf}_{\xx} & \hat{\Phibf}_{\xy}\\
\hat{\Phibf}_{\ux} & \hat{\Phibf}_{\uy}
}
}_{\infty} \leq \epsilon^{-1}.
\end{align*}
\end{corollary}

\begin{proof}
Similar to the proof of \cor{robust-IOP}, we know $\Deltabf \in \RHinf$ as its norm is bounded. By the small gain theorem and \cor{robust-of-SLS-general}, $\Psibf \in \RHinf$ for all $\Deltabf \in \Dcal_\epsilon$ if and only if
\begin{align*}
\norm{
\mat{
\hat{\Phibf}_{\xx} & \hat{\Phibf}_{\xy}\\
\hat{\Phibf}_{\ux} & \hat{\Phibf}_{\uy}
}
}_{\infty} \leq \epsilon^{-1},
\end{align*}
which concludes the proof.
\end{proof}

Similar to \cor{robust-sf-SLS}, \ref{cor:robust-of-SLS-controller-condition}, and \ref{cor:robust-of-SLS-general}, we can derive the following condition for the nominal IOP controller that still stabilizes a perturbed plant using \lem{perturbed-S}. The proof is mostly the same as the proof of \cor{robust-IOP} and omitted.

\begin{corollary}
For the realization in \fig{realization-G-K}, let $\{\hat{\Ybf}, \hat{\Wbf}, \hat{\Ubf}, \hat{\Zbf} \}$ satisfy \eqn{IOP-constraints}. Consider an additive perturbation $\DeltaG \in \RHinf$ on the plant such that $\Gbf(\Deltabf_{\Gbf}) = \Gbf + \DeltaG$. Then the nominal controller
\begin{align*}
\hat{\Kbf} = \hat{\Ubf} \hat{\Ybf}^{-1}
\end{align*}
internally stabilizes the perturbed system if and only if
\begin{align*}
(I - \DeltaG \hat{\Ubf})^{-1} \in \RHinf.
\end{align*}
\end{corollary}

\section{Discussion}\label{sec:discussion}

Although the conditions derived from \lem{R-S} and \lem{perturbed-S} in this work can unify a large portion of robust controller synthesis results in the literature, there are still results beyond its scope. In particular, the conditions rely on some other procedure to ensure the perturbed stability matrix is still in $\RHinf$ for the whole uncertainty set $\Dcal$, or condition \eqn{robust-S-general}.
Different $\Dcal$ may invoke different theorems. For example, while the small gain theorem works for ball-like uncertainty set, there is a line of research on bounded perturbation of transfer function coefficients, which builds upon the Kharitonov's Theorem \cite{kharitonov1978asymptotic}.
Kharitonov's Theorem suggests that robust stability over the whole uncertainty set $\Dcal$ can be achieved by the stability of $4$ elements within $\Dcal$. It is leveraged by \cite{bernstein1990robust,bernstein1992robust,shafai1992robust} to synthesize robust controllers and further generalized by \cite{chapellat1989generalization,barmish1988generalization} for larger classes of coefficient-perturbed transfer functions.

In contrast to the realization-centric perspective adopted in this work, there are also alternative approaches for system analysis using the integral quadratic constraints \cite{megretski1997system} or interval analysis \cite{patre2010robust}. Those methods also derive robust results, while we argue that the realization-centric perspective is more straightforward and the proofs are much simpler.

There is still a plenty of robust results that could be derived or generalized from \lem{R-S} and \lem{perturbed-S} but not included here, such as \cite[Proposition 1]{chen2019robust}, \cite[Lemma 2]{tsiamis2020sample}, \cite[Theorem III.5]{matni2020robust}. We encourage the reader to explore those diverse results and unify them under this paper's approach.

\section{Conclusion}\label{sec:conclusion}

We derived sufficient and necessary robust stability conditions from the realization-stability lemma and established the general robust controller synthesis problem. Several existing robust results can be derived from the conditions, including $\mu$-synthesis, robust primal-dual Youla parameterization, robust IOP, and robust SLS. Further, we demonstrate how to easily obtain new robust results for output-feedback SLS and IOP from the conditions. Along with some discussions about this work's scope and how it complements/facilitates other existing work on robust controller synthesis, we present a unified approach to analyze and synthesize robust systems.

\bibliographystyle{IEEEtran}
\bibliography{Test}

\end{document}